\pdfoutput=1
\documentclass[amsmath,secnumarabic,floatfix,amssymb,nofootinbib,nobibnotes,letterpaper,11pt,tightenlines]{revtex4}

\usepackage[T1]{fontenc}
\usepackage{geometry}
\usepackage{latexsym}
\usepackage{amsmath}
\usepackage{amssymb}
\usepackage{amscd}
\usepackage{amsthm}
\usepackage{mathtools}
\usepackage{cancel}
\usepackage[dvipsnames]{xcolor}
\usepackage[export]{adjustbox}

\usepackage{newtxtext,newtxmath}
\AtBeginDocument{%
  \mathchardef\standardeq=\mathcode`=
  \mathchardef\standardless=\mathcode`<
  \mathchardef\standardgreater=\mathcode`>
  \mathcode`="8000
  \mathcode`<"8000
  \mathcode`>"8000
}
\begingroup\lccode`~`=\lowercase{\endgroup
  \def~}{\mathrel{\mspace{0.936mu}\standardeq\mspace{0.936mu}}}
\begingroup\lccode`~`<\lowercase{\endgroup
  \def~}{\mathrel{\mspace{0.153mu}\standardless\mspace{0.153mu}}}
\begingroup\lccode`~`>\lowercase{\endgroup
  \def~}{\mathrel{\mspace{0.153mu}\standardgreater\mspace{0.153mu}}}

\usepackage{graphicx}
\usepackage[percent]{overpic}
\usepackage{units}

\usepackage{algorithm}
\usepackage{algorithmicx}
\usepackage{algpseudocode}

\usepackage{enumitem}
\setlist{nosep}

\def\figdir{Pictures/}
\graphicspath{{\figdir}{Pictures/structure}{Pictures/permutations}{Pictures/midpoint-lemma}}

\usepackage{aliascnt} 
\newcommand{\newautotheorem}[3] 
{
\newaliascnt{#1}{#2}
\newtheorem{#1}[#1]{#3}
\aliascntresetthe{#1}
\expandafter\def\csname #1autorefname\endcsname{%
#3%
}%
}

\newtheorem{theorem}{Theorem}

\newautotheorem{lemma}{theorem}{Lemma}
\newautotheorem{proposition}{theorem}{Proposition}
\newautotheorem{corollary}{theorem}{Corollary}

\theoremstyle{definition}
\newautotheorem{definition}{theorem}{Definition}

\newautotheorem{conjecture}{theorem}{Conjecture}
\newautotheorem{remark}{theorem}{Remark}
\newautotheorem{claim}{theorem}{Claim}

\newcommand{\KronDelta}{\updelta} 

\newcommand{\Rog}{\operatorname{R}_{\mathrm{g}}^2}

\newcommand{\R}{\mathbb{R}}

\newcommand{\Var}{\operatorname{Var}}

\newcommand{\graphG}{\mathbf{G}}
\newcommand{\edgesE}{\mathbf{e}}
\newcommand{\verticesV}{\mathbf{v}}
\newcommand{\edge}{e}
\newcommand{\vertex}{v}
\newcommand{\head}{\operatorname{head}}
\newcommand{\tail}{\operatorname{tail}}

\newcommand{\ceq}{\coloneqq}

\def\co{\colon\thinspace}

\let\mgp=\marginpar \marginparwidth18mm \marginparsep1mm
\def\marginpar#1{\mgp{\raggedright\tiny #1}}

\let\lbl=\label
\def\label#1{\lbl{#1}\ifinner\else\marginpar{\ref{#1} #1}\ignorespaces\fi}

\newcommand{\edgegroup}{\mathcal{E}}
\newcommand{\vertexgroup}{\mathcal{V}}
\newcommand{\positiongroup}{X}
\newcommand{\displacementgroup}{W}
\newcommand{\comcloud}{M}
\newcommand{\parentcloud}{P}

\DeclarePairedDelimiterXPP{\pars}[1]{\mathop{}}{\lparen}{\rparen}{}{#1}

\DeclarePairedDelimiterXPP{\abs}[1]{\mathop{}}{\lvert}{\rvert}{}{#1}
\DeclarePairedDelimiterXPP{\norm}[1]{\mathop{}}{\lVert}{\rVert}{}{#1}
\DeclarePairedDelimiterXPP{\seminorm}[1]{\mathop{}}{\lbrack}{\rbrack}{}{#1}
\DeclarePairedDelimiterXPP{\inner}[1]{\mathop{}}{\langle}{\rangle}{}{#1}
\DeclarePairedDelimiterXPP{\brackets}[1]{\mathop{}}{\lbrack}{\rbrack}{}{#1}
\DeclarePairedDelimiterXPP{\braces}[1]{\mathop{}}{\lbrace}{\rbrace}{}{#1}

\DeclarePairedDelimiterXPP{\intervalcc}[1]{\mathop{}}{\lbrack}{\rbrack}{}{#1}
\DeclarePairedDelimiterXPP{\intervalco}[1]{\mathop{}}{\lbrack}{\rparen}{}{#1}
\DeclarePairedDelimiterXPP{\intervaloc}[1]{\mathop{}}{\lparen}{\rbrack}{}{#1}
\DeclarePairedDelimiterXPP{\intervaloo}[1]{\mathop{}}{\lparen}{\rparen}{}{#1}

\DeclarePairedDelimiterXPP{\set}[2]{\mathop{}}{\lbrace}{\rbrace}{}{#1\,\delimsize\vert\,\mathopen{}#2}

\let\dot\undefined

\DeclarePairedDelimiterXPP{\dot}[2]{\mathop{}}{\langle}{\rangle}{}{#1,#2}

\DeclarePairedDelimiterXPP{\floor}[1]{\mathop{}}{\lfloor}{\rfloor}{}{#1}
\DeclarePairedDelimiterXPP{\ceil}[1]{\mathop{}}{\lceil}{\rceil}{}{#1}

\DeclareDocumentCommand{\converges}{ o }{
	\mathbin{%
		\IfValueTF{#1}{%
			\mathrel{\vbox{\offinterlineskip\ialign{%
				\hfil##\hfil\cr
				$\scriptscriptstyle#1$\cr
				$-\!\!\!-\!\!\!\rightarrow$\cr
			}}}
		}{%
			-\!\!\!-\!\!\!\rightarrow
		}%
	}%
}

\usepackage[
	backref  = false%
	,pagebackref = false%
	,bookmarks  = true%
	,bookmarksdepth=3%
]{hyperref}
\usepackage[all]{hypcap}

\begin{document}
\title{On the average squared radius of gyration \\of a family of embeddings of subdivision graphs}
\author{Jason Cantarella}
\altaffiliation{Mathematics Department, University of Georgia, Athens, GA, USA}
\noaffiliation
\author{Henrik Schumacher}
\altaffiliation{Mathematics Department, University of Georgia, Athens, GA, USA}
\noaffiliation
\author{Clayton Shonkwiler}
\altaffiliation{Department of Mathematics, Colorado State University, Fort Collins, CO, USA}
\noaffiliation



\begin{abstract}
Suppose we have an embedding of a graph $\graphG$ created by subdividing the edges of a simpler graph $\graphG'$. The edges of $\graphG$ can be divided into subsets which join pairs of ``junction'' vertices in $\graphG'$. The displacement vectors of the edges in each subset sum to the displacement between junctions. We can construct a family of embeddings of $\graphG$ with the same junction positions by rearranging the displacements in each group. In this paper, we show that the average (squared) radius of gyration of these embeddings is given by a simple formula involving a weighted (squared) radius of gyration of the positions of the junctions and the sum of the squares of the lengths of the edges of $\graphG$ and $\graphG'$. This ensemble of graph embeddings arises naturally in polymer science.
\end{abstract}
\date{\today}
\maketitle

\section{Introduction}
In this paper, we consider some geometric properties of a special family of graph embeddings.\footnote{``Graph embedding'' is a term of art---see, for example, the survey~\cite{GoyalFerrara2018}---which refers to a mapping from the vertex set of a graph to a vector space. This is not necessarily a topological embedding, as there is no assumption that the mapping is injective.} Let $\graphG$ be a directed graph with $\verticesV$ vertices and $\edgesE$ edges. An embedding of $\graphG$ into $\R^d$ is given by a choice of positions $X = (x_1,\dotsc,x_{\verticesV}) \in (\R^d)^\verticesV$ for the vertices of $\graphG$.
These vertex positions (and the directions on the edges) determine edge displacements $W = (w_1, \dotsc, w_{\edgesE}) \in (\R^d)^\edgesE$: if the head and tail of edge $i$ are vertices $j$ and $k$, then $w_{i} = x_j - x_k$.

Suppose that $\graphG$ is a subdivision of some $\graphG'$, which has $\verticesV'$ vertices and $\edgesE'$ edges, so that each edge of $\graphG'$ is subdivided into $n$ pieces as in~\autoref{fig:structure}. We first note that an embedding $X$ of $\graphG$ immediately determines an embedding $X'$ of $\graphG'$. Further, if we divide the edges of $\graphG$ into $\edgesE'$ sets of $n$ edges, denoting the $j$th member of the $i$th group by $w_{i,j}$, then any permutation $\sigma = (\sigma_1,\dotsc,\sigma_{\edgesE'})$, $\sigma_i \in S_{n}$ of the $n \cdot \edgesE'$ displacement vectors in $W$ that preserves each group of $n$ yields an embedding $X^\sigma$ of $\graphG$ which determines the~\emph{same} embedding $X'$ of $\graphG'$, as shown in~\autoref{fig:permutation}. We will denote the group of such permutations by $S$.

\begin{figure}[t]
\hfill
\begin{overpic}[width=0.4\textwidth]{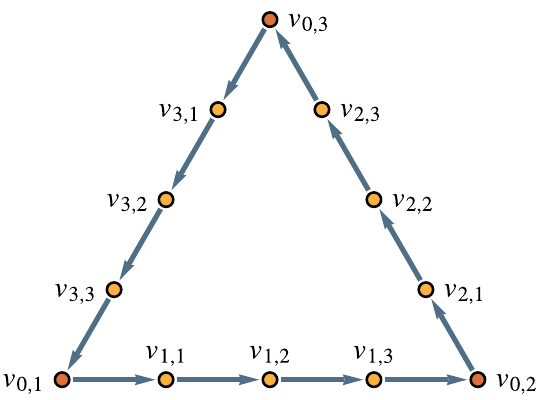}
\end{overpic}
\hfill
\begin{overpic}[width=0.4\textwidth]{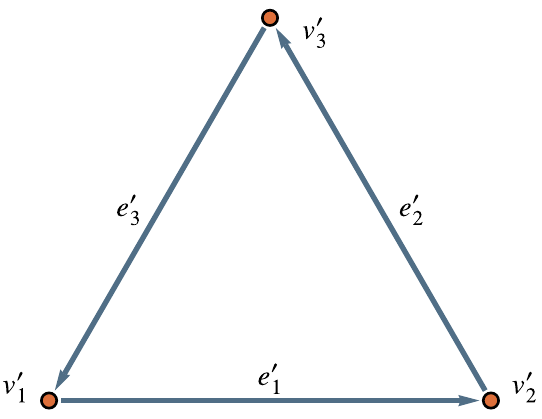}
\end{overpic}
\hfill
\hphantom{.}
\caption{The graph $\graphG$ (at left) is a subdivision of the graph $\graphG'$ at right. The vertices $\vertex_{i,j}$ and edges $\edge_{i,j}$ of $\graphG$ are numbered to correspond with vertices and edges of $\graphG'$; each $\vertex_{0,j}$ corresponds to a vertex $\vertex'_j$ of $\graphG'$, while vertices $\vertex_{i,1}, \dotsc, \vertex_{i,n-1}$ are those created by subdividing edge $\edge'_i$ of $\graphG'$ into $n$ new edges. The edges $\edge_{i,j}$ of $\graphG$ aren't labeled in the picture, but are constructed so that $\edge_{i,1}, \dotsc, \edge_{i,n}$ are the edges created by subdividing $\edge'_i$.}
\label{fig:structure}
\end{figure}

\begin{figure}[t]%
	\includegraphics[width=0.9\textwidth]{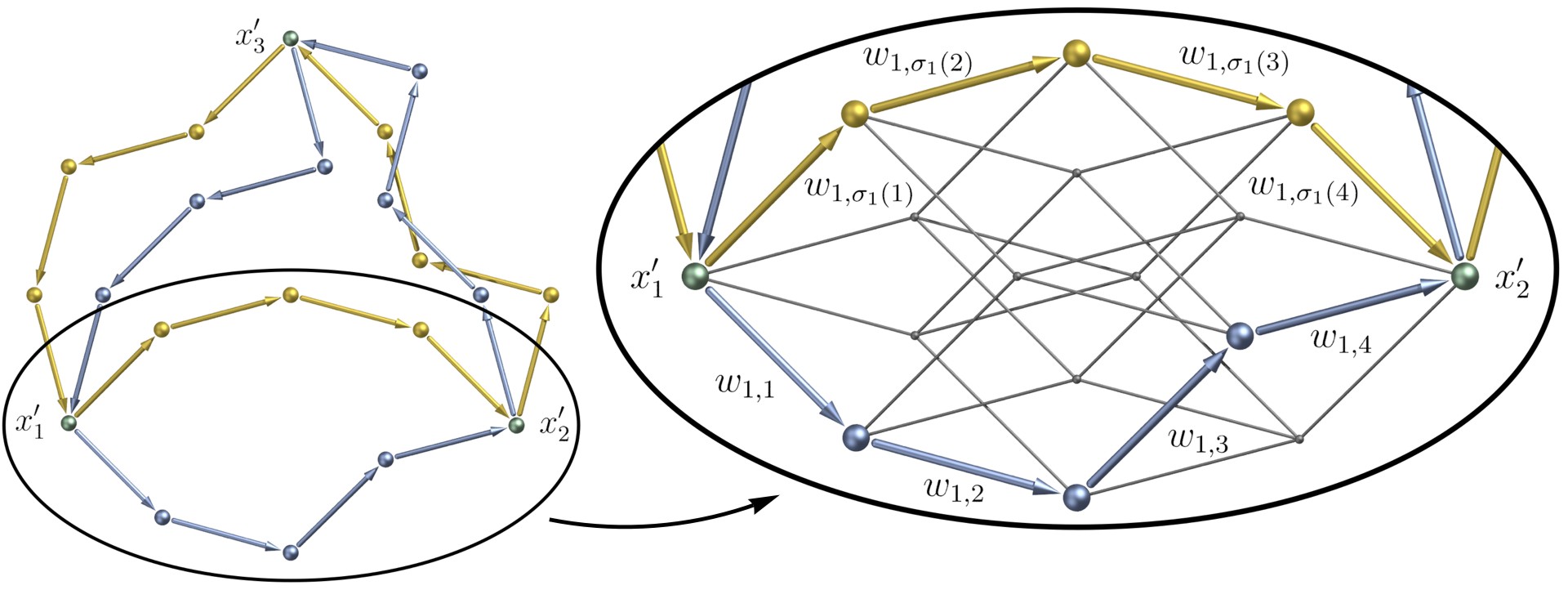}%
	\caption{Above left we see two different embeddings of the cycle graph $\graphG$, which is a subdivision of the triangle graph $\graphG'$ from~\autoref{fig:structure} with $n=4$. The edges of $\graphG$ are divided into 3 groups of $4$ edges, each corresponding to a single edge of $\graphG'$. The yellow and blue embeddings of $\graphG$ are generated by permuting the displacement vectors within each group. As we see in the inset graphic of the bottom arc (at right), there are many (in fact,~$n!$) different paths that $\graphG$ may take along each edge of $\graphG'$. However, the set of vertex positions along these paths is highly structured, and any such permutation gives rise to the same embedding of $\graphG'$.}
	\label{fig:permutation}
\end{figure}

This gives rise to the following question: to what extent is the average geometry of the $X^\sigma$ determined by $X'$? We are particularly interested in the (squared) \emph{radius of gyration}
\[
	\Rog(X)
	\ceq
	\frac{1}{\verticesV} \sum_{i=1}^{\verticesV}
	\norm{x_i - \mu }^2
	=
	\frac{1}{\verticesV} \sum_{i=1}^{\verticesV} \norm{x_i}^2 - \norm{\mu}^2
	,
	\quad \text{where} \quad
	\mu
	\ceq
	\frac{1}{\verticesV} \sum_{j=1}^{\verticesV} x_j
\]
of these embeddings, and prove as our main theorem the appealing formula:

\begin{theorem}
\label{thm:symmetrization formula}
The average radius of gyration
\[
\begin{aligned}
	\frac{1}{\# S} \sum_{\sigma \in S} \Rog(X^\sigma)
	&=
	\Rog\pars[\Big]{ X',\deg + \frac{2}{n-1} }
	+
	\frac{ \pars{n+1} \pars{2 \, \verticesV - n} }{ 12 \, \verticesV^2 } \norm{W}^2
	-
	\frac{ \pars{n+1} \pars{2 \, \verticesV - 1} }{ 12 \, \verticesV^2 } \norm{W'}^2
	,
\end{aligned}
\]
where $\Rog\pars[\big]{X',\deg + \frac{2}{n-1}}$ is a reweighted radius of gyration (see \autoref{defn:radius of gyration} below) where each vertex is weighted by its degree plus $\frac{2}{n-1}$,
$\norm{W}^2 = \sum_{i=1}^{\edgesE'} \sum_{j=1}^{n} \norm{w_{i,j}}^2$
and
$\norm{W'}^2 = \sum_{i=1}^{\edgesE'\!\!} \norm{w'_i}^2$.
\end{theorem}
These particular ensembles of graph embeddings are motivated by polymer science, where the embeddings $X$ of network polymers are random variables determined by a probability distribution on the edge displacements $W$ (conditioned on membership in the subspace of acceptable displacements). The theory of these~\emph{topological polymers} was first discussed by James, Guth, and Flory~\cite{James1943,James:1947hp,FloryPaulJ1969Smoc} and called \emph{phantom network theory}. Recently, polymers with complicated predetermined topologies have actually been synthesized~\cite{Suzuki:2014fo,Tezuka:2017gh}, leading to renewed interest in extending and understanding the classical theory. In phantom network theory, the distribution on the edges~$W$ is Gaussian, and in particular is invariant under the permutations $\sigma$ described above. More modern versions of the theory replace the Gaussian edge distribution with something more physically motivated, such as a fixed edgelength (freely jointed networks) or an energetic potential. But these distributions are still invariant under the permutation action on edges in a subdivision. Since the radius of gyration can be directly measured experimentally (for instance by small angle neutron scattering~\cite{WeiHore2021}), understanding the distribution of radii of gyration is extremely important in polymer science.

In the companion paper~\cite{contractionfactors}, we use \autoref{thm:symmetrization formula} to compute the exact expectation of radius of gyration for all subdivided graphs in phantom network theory. This quantity turns out to depend only on the underlying graph $\graphG'$ and on the number $n$ of subdivisions.

\section{Notation and background}

In physics, one usually defines a weighted point cloud to be a finite collection of vectors in $\R^d$ with corresponding weights. However, this can introduce some notational difficulties if points coincide. So we are a bit more formal here.
\begin{definition}
A~\emph{weighted point cloud} is a finite~\emph{index set} $V = \{v_1, \dots, v_n\}$ together with a~\emph{position function} $X \co V \rightarrow \R^d$ and a~\emph{weight function} $\varOmega \co V \to \R^+$. 
\end{definition}
We denote such a cloud by $(X,\varOmega)$ or just by $X$ if $\varOmega$ is a constant function. 
We can now give the usual definitions:
\begin{definition}
\label{defn:radius of gyration}
We define the~\emph{total weight} $\abs{\varOmega} = \sum_{v \in V} \varOmega(v)$.
The~\emph{center of mass} or~\emph{expectation} of a weighted point cloud is given by $\mu(X,\varOmega) = \frac{1}{\abs{\varOmega}} \sum_{v \in V} \varOmega(v) x(v)$. The (weighted) \emph{radius of gyration} or~\emph{variance} $\Rog(X,\varOmega)$ is given by any of the three equivalent expressions
\begin{align}
	\Rog(X,\varOmega)
	&\ceq\frac{1}{2\abs{\varOmega}^2} \sum_{i \in 1}^n \sum_{j \in 1}^n \omega_i \, \omega_j \norm{x_i - x_j}^2
	\label{eq:DefRog1}
	\\
	&= \frac{1}{\abs{\varOmega}} \sum_{i \in 1}^n \omega_i \norm{x_i - \mu(X,\varOmega)}^2
	\label{eq:DefRog2}
	\\
	&= \frac{1}{\abs{\varOmega}} \left( \sum_{i \in 1}^n \omega_i \norm{x_i}^2 \right) - \norm{\mu(X,\varOmega)}^2.
	\label{eq:DefRog3}
\end{align}
where $x_i := X(v_i)$ and $\omega_i = \varOmega(v_i)$.
\end{definition}
The equality between the first two lines is standard in physics while the equality between the second two is standard in probability, where $\Rog(X,\varOmega)$ is the scalar variance of the vector-valued random variable $X$ on the probability space $V$ where each $v \in V$ has probability $\varOmega(v)/\abs{\varOmega}$. The proofs are the usual ones. We note that rescaling the weights does not change either $\Rog(X,\varOmega)$ or $\mu(X,\varOmega)$.
Thus, when the weights in $\varOmega$ are all equal, we may assume without loss of generality that all $\varOmega(x) = 1$. In this case, we omit the $\varOmega$ in $\Rog(X,\varOmega)$ and $\mu(X,\varOmega)$, writing $\Rog(X)$ and $\mu(X)$.

We will need the following property of $\Rog$ which in principle follows easily from Eve's law. Since we are using a generalization of variance for the vector-valued random variate $X$, we provide an elementary proof in the Appendix for interested readers to check that everything goes through as it does in the usual case (cf.~\cite[(1.5b)]{ChanGolubLeveque1983} or \cite{oneill2016}).

\begin{lemma}
\label{lem:splitting formula for Rog}
Suppose that we have an index set $V$, a single position function $X \co V \to \R^d$ and a finite set of  weight functions $\varOmega_i \co V \to \R^+$, where $i \in 1, \dotsc, m$. Further, let $\varOmega \co V \to \R^+$ be the weight function defined by $\varOmega = \sum_{i=1}^m \varOmega_{i}$. Then 
\begin{equation}
\label{eq:splitting formula}
\Rog(X,\varOmega) =   \sum_{i=1}^m  \frac{\abs{\varOmega_i}}{\abs{\varOmega}} \, \Rog(X,\varOmega_i)
	+
	\frac{1}{2 \abs{\varOmega}^2} \sum_{i=1}^m \sum_{j=1}^m \abs{\varOmega_i} \abs{\varOmega_j} \norm{\mu(X,\varOmega_i) - \mu(X,\varOmega_j)}^2.
\end{equation}
\end{lemma}

We will deal with connected, oriented graphs $\graphG$ and $\graphG'$ (allowing loop edges and multiple edges joining the same pair of vertices), where $\graphG$ is constructed from $\graphG'$ by subdividing each edge of $\graphG'$ into $n$ sub-edges.
We assume that $\graphG'$ has vertices $\{\vertex'_1, \dotsc, \vertex'_{\verticesV'}\}$ and edges $\braces{\edge'_1, \dotsc, \edge'_{\edgesE'}}$. The vertices of $\graphG$ are denoted $\vertex_{i,j}$ where $\vertex_{0,j} = \vertex'_j$ is a vertex of $\graphG'$ and $\vertex_{i,j}$ is the $j$th new vertex created by subdividing edge $\edge'_i$ of $\graphG'$.
Note that either $i = 0$ and $j \in \braces{1, \dotsc, \verticesV'}$ or $i \in \braces{1, \dotsc, \edgesE'}$ and $j \in \braces{1, \dotsc, n-1}$. 

We use $\vertexgroup$ to denote the set of vertices of $\graphG$ and use $\vertexgroup'$ for the set of vertices of $\graphG'$.
Further, let ${\vertexgroup_i \ceq \braces{v_{i,1},\dotsc,v_{i,n-1}}}$ be the set of vertices with first index~$i \in \braces{1,\dotsc,\edgesE'}$, and let ${\vertexgroup_0 \ceq \braces{v_{0,1}, \dotsc, v_{0,\edgesE'}}}$ be the vertices originating from the structure graph.
The edges of $\graphG$ are denoted $\edge_{i,j}$, where this is the $j$th edge created by subdividing~$\edge'_i$.
Note that $i \in \braces{1, \dotsc, \edgesE'}$ and $j \in \braces{1, \dotsc, n}$.
We let $\edgegroup_i \ceq \braces{e_{i,1},\dotsc,e_{i,n}}$ be the set of edges with first index~$i$.

Since each graph is oriented, there are maps $\head$ and $\tail$ giving the indices of the incoming and outgoing vertices associated to each edge index, so that $\edge_{i,j}$ joins $\vertex_{\tail(i,j)}$ to $\vertex_{\head(i,j)}$ and $\edge'_i$ joins $\vertex'_{\tail(i)}$ to $\vertex'_{\head(i)}$.
By construction we have
\begin{equation}
\begin{aligned}
	\tail(i,1) &= (0,\tail(i))
	,
	&
	\tail(i,j) &= (i,j-1) \;\; \text{for $j \in \braces{2,\dotsc,n}$}
	,
	\\
	\head(i,n) &= (0,\head(i))
	,
	&
	\head(i,j) &= (i,j) \;\; \text{for $j \in \braces{1,\dotsc,n-1}$}
	.
\end{aligned}
\label{eq:HeadsAndTails}
\end{equation}

Embeddings of $\graphG$ and $\graphG'$ in $\R^d$ are really weighted point clouds with position functions $X \co \vertexgroup \to \R^d$ and $X' \co \vertexgroup' \to \R^d$ and all weights equal to 1. The position functions $\positiongroup_{i} : \vertexgroup_i \to \R^d$ for $i \in \{0, \dotsc, \edgesE'\}$ defined by restricting $\positiongroup$ to each $\vertexgroup_i$ construct $\edgesE + 1$ (smaller) weighted point clouds $\positiongroup_i$, again with all weights equal to 1. We let $x_{0,j} := \positiongroup_0(v_{0,j}) = \positiongroup(v_{0,j})$ for $j \in \{1, \dotsc,\verticesV'\}$, and $x_{i,j} := \positiongroup_{i}(v_{i,j}) = \positiongroup(v_{i,j})$ for $i \in \{1, \dotsc, \edgesE'\}$ and $j \in \{1, \dotsc, n-1\}$. Similarly, $x'_j := X'(v'_j)$ for $j \in \{1, \dotsc,\verticesV'\}$.

We say that $X$ is~\emph{compatible} with $X'$ if $x_{0,j} = x'_j$ for $j \in 1, \dots, \verticesV'$. If we identify $\vertexgroup_0$ with $\vertexgroup'$, we then have $\positiongroup_0 = \positiongroup'$ on this set. 

Since the graphs are oriented, these position functions give rise to corresponding ``displacement'' functions $\displacementgroup : \edgegroup \to \R^d$ and $W' : \edgegroup' \to \R^d$ given by 
\[
\displacementgroup(\edge_{i,j}) = X(v_{\head(i,j)}) - X(v_{\tail(i,j)}), \quad\text{and}\quad \displacementgroup'(\edge'_j) = X'(v'_{\head(j)}) - X(v'_{\tail(j)}).
\label{eq:definition of displacement vectors}
\]
We define $\displacementgroup_i \co \edgegroup_i \to \R^d$ for $i \in \{1, \dotsc, \edgesE'\}$ by restricting $\displacementgroup$ to $\edgegroup_i$, and let $w_{i,j} := \displacementgroup_i(\edge_{i,j}) = \displacementgroup(\edge_{i,j})$ and $w'_i := \displacementgroup'(\edge'_i)$ for $i \in \{1, \dotsc, \edgesE'\}$ and $j \in \{1, \dotsc, n\}$. If we think of them as position functions, $\displacementgroup$, $\displacementgroup_i$, and $\displacementgroup'$ are  all also weighted point clouds with all weights equal to 1.

We now express $\Rog(\positiongroup)$ in terms of properties of the smaller point clouds $\positiongroup_i$.

\begin{proposition}
\label{prop:splitting formula for rog}
Suppose $\graphG'$ is a graph with $\verticesV'$ vertices and $\edgesE'$ edges and $\graphG$ is a graph with $\verticesV$ vertices and $\edgesE$ edges created by subdividing each edge of $\graphG'$ into $n$ pieces. Further, suppose that $X : \vertexgroup \to \R^d$ is an embedding of $\graphG$, and $X' : \vertexgroup' \to \R^d$ is the corresponding compatible embedding of $\graphG'$. Then
\[
\begin{aligned}
	\Rog(X)
	&=
	\frac{n-1}{\verticesV} \sum_{i=1}^{\edgesE'\!\!}  \Rog(\positiongroup_i)
	+
	\frac{\verticesV'}{\verticesV} \Rog(X')
	+
	\\
	&\qquad
    +\frac{\pars{n-1}^2}{2 \, \verticesV^2} \sum_{i = 1}^{\edgesE'\!\!} \sum_{j = 1}^{\edgesE'\!\!} \norm{\mu(\positiongroup_i) - \mu(\positiongroup_j)}^2
	+
	\frac{\pars{n-1} \, \verticesV'}{\verticesV^2} \sum_{i=1}^{\edgesE'\!\!} \norm{\mu(\positiongroup_i) - \mu(X')}^2
	.
\end{aligned}
\]
\end{proposition}

\begin{proof}
Our goal is to apply~\autoref{lem:splitting formula for Rog}. We start by defining weight functions $\varOmega_0, \dots, \varOmega_{\edgesE'}$ on $\vertexgroup$, where $\varOmega_{i'}(v_{i,j}) := \KronDelta_{i',i}$. Observe $\varOmega := \sum_{i'} \varOmega_{i'}$ has $\varOmega(v_{i,j}) = \sum_{i'=0}^{\edgesE'} \varOmega_{i'}(v_{i,j}) = 1$ for each $v_{i,j}$. Further, the total weights $\abs{\varOmega} = \verticesV$, $\abs{\varOmega_0} = \verticesV'$, and $\abs{\varOmega_1} = \cdots = \abs{\varOmega_{\edgesE'}} = n-1$. 

It follows immediately that $\Rog(X,\varOmega_i) = \Rog(\positiongroup_i)$ and $\mu(X,\varOmega_i) = \mu(\positiongroup_i)$; in particular that $\Rog(X,\varOmega_0) = \Rog(\positiongroup_0) = \Rog(X')$ and $\mu(X,\varOmega_0) = \mu(\positiongroup_0) = \mu(X')$. Applying~\autoref{lem:splitting formula for Rog} then yields the result.
\end{proof}

\section{Symmetrizing over rearrangements of the entries in each $\displacementgroup_i$}

We start with a definition:

\begin{definition}
\label{defn:big group of permutations}
Let $S = (S_n)^{\edgesE'}$ be the product of $\edgesE'$ copies of the permutation group $S_n$ of $n$ elements. The group $S$ acts on edge (or displacement vector) indices in the expected way: if $\sigma \in S$ is given by $(\sigma_1,\dotsc, \sigma_{\edgesE'})$ then $\sigma(i,j) = (i,\sigma_i(j))$. That is, $\sigma$ permutes indices within each subdivided edge.
\end{definition}

\begin{proposition}
Suppose that $X \co \vertexgroup \to \R^d$ and $X' \co \vertexgroup \to \R^d$ are compatible and $\sigma \in S$. Then $X^\sigma \co \vertexgroup \to \R^d$ defines a point cloud with index set $\vertexgroup$ given by
\begin{equation}
	x^\sigma_{0,j} \ceq x'_j,
	\quad \text{and} \quad
	x^\sigma_{i,j} \ceq x^\sigma_{0,\tail(i)} + \sum_{k=1}^j w_{\sigma(i,k)}
	\quad
	\text{for $j \in \braces{1,\dotsc,n-1}$.}
	\label{eq:definition of xsigmaij}
\end{equation}
This position function is compatible with $X'$ and has displacement vectors $w^{\sigma}_{i,j} \ceq w_{\sigma(i,j)}$.
\label{prop:permutations preserve embeddability}
\end{proposition}

\begin{proof}
	We have to prove that $x^\sigma_{\head(i,j)} - x^\sigma_{\tail(i,j)} = w_{\sigma(i,j)}$ for all $i$ and $j$.
	For $j \in \braces{2,\dotsc,n-1}$,
	we know from \eqref{eq:HeadsAndTails} that
	$\head(i,j) = (i,j)$ and $\tail(i,j) = (i,j-1)$.
	Thus \eqref{eq:definition of xsigmaij} leads to
	\begin{align*}
		x^\sigma_{\head(i,j)} - x^\sigma_{\tail(i,j)}
		&=
		x^\sigma_{i,j} - x^\sigma_{i,j-1}
		\\
		&=
		\pars*{
			x^\sigma_{0,\tail(i)} + \sum_{k=1}^j w_{\sigma(i,k)}
		}
		-
		\pars*{
			x^\sigma_{0,\tail(i)} + \sum_{k=1}^{j-1} w_{\sigma(i,k)}
		}
		=
		w_{\sigma(i,j)}
		.
	\end{align*}
	For $j = 1$ we have $\head(i,1) = (i,1)$ and $\tail(i,1) = (0,\tail(i))$, thus
	\begin{equation*}
		x^\sigma_{\head(i,1)} - x^\sigma_{\tail(i,1)}
		=
		x^\sigma_{(i,1)} - x^\sigma_{(0,\tail(i))}
		=
		\pars[\big]{
			x^\sigma_{0,\tail(i)} + w_{\sigma(i,1)}
		}
		-
		x^\sigma_{0,\tail(i)}
		=
		w_{\sigma(i,j)}
		.
	\end{equation*}
	Finally, for $j = n$ we recall that $\head(i,n) = (0,\head(i))$ and $\tail(i,n) = (i,n-1)$.
	Moreover,
	\begin{equation*}
		x^\sigma_{0,\head(i)}
		= x^\sigma_{0,\tail(i)} + \sum_{k=1}^n w_{i,k}
		= x^\sigma_{0,\tail(i)} + \sum_{k=1}^n w_{i,\sigma_i(k)}
		= x^\sigma_{0,\tail(i)} + \sum_{k=1}^n w^\sigma_{i,k}
	\end{equation*}
	because summation is commutative, so the sum of the $w_{i,k}$ in the second term is the same as the sum of the permuted $w_{i,\sigma_i}(k)$ in the third. Thus, we finally obtain
	\begin{align*}
		x^\sigma_{\head(i,n)} - x^\sigma_{\tail(i,n)}
		&=
		x^\sigma_{(0,\head(i))} - x^\sigma_{(i,n-1)}
		\\
		&=
		\pars*{
			x^\sigma_{0,\tail(i)} + \sum_{k=1}^n w^\sigma_{i,k}
		}
		-
		\pars*{
			x^\sigma_{0,\tail(i)} + \sum_{k=1}^{n-1} w_{\sigma(i,k)}
		}
		=
		w_{\sigma(i,n)}
		.
		\qedhere
	\end{align*}
\end{proof}

In many polymer models, the probability distribution on $W$ (even when conditioned on the overall graph type) is exchangeable among the edges in each $\edgegroup_i$.
This means that if $\sigma \in S$, all embeddings $X^\sigma$ of $\graphG$ are equally probable.
Hence, the expected radius of gyration of $\graphG$ is the same as the expectation of the average radius of gyration of such $X^\sigma$, $\sigma \in S$:
\begin{align*}
	E \!\pars[\big]{ \Rog(X) }
	=
	E \!\pars[\bigg]{ \frac{1}{\# S}  \sum_{\sigma \in S} \Rog(X^\sigma) }
	.
\end{align*}
This observation originally motivated us to try to find a simple formula for the (finite) average over permutations on the right hand side, in the hope that such a formula would make the right-hand expectation easier to compute than the left-hand one. This is indeed the case.

\begin{definition}
Given an element $\sigma \in S$ and indices $i \in \braces{1, \dotsc, \edgesE'}$, and $j, k \in \braces{1, \dotsc, n}$, we define the~\emph{indicator variable}
\begin{align}
	c(i,j,k,\sigma)
	\ceq
	\begin{cases}
		1, & \sigma^{-1}_i(k) \leq j, \\
		0, & \sigma^{-1}_i(k) > j.
	\end{cases}
\end{align}
\end{definition}

\begin{lemma} \label{lem:indicator variables}
We have
\[
	\sum_{k=1}^j w_{\sigma(i,k)} = \sum_{k=1}^n c(i,j,k,\sigma) \, w_{i,k}
	,
	\quad
	\text{and}
	\quad
	x^\sigma_{i,j} = x^\sigma_{0,\tail(i)} + \sum_{k=1}^n c(i,j,k,\sigma) \, w_{i,k}
	.
\]
\end{lemma}

\begin{proof} Clearly, $\sum_{k=1}^j w_{\sigma(i,k)} = \sum_{k=1}^n u(k,j) \, w_{\sigma(i,k)}$ where $u(k,j) = 1$ if $k \leq j$ and $0$ otherwise. Since we are summing over all $k \in \braces{1, \dotsc, n}$, we get the same answer if we permute the $k$, replacing each $k$ with $\sigma_i^{-1}(k)$. This yields
\[
\begin{aligned}
\sum_{k=1}^n u(k,j) \, w_{\sigma(i,k)} &= \sum_{k=1}^n u(\sigma^{-1}_i(k),j) \, w_{\sigma(i,\sigma_i^{-1}(k))} \\
&= \sum_{k=1}^n u(\sigma^{-1}_i(k),j) \, w_{i,\sigma_i \sigma_i^{-1}(k)} =
\sum_{k=1}^n u(\sigma^{-1}_i(k),j) \, w_{i,k}.
\end{aligned}
\]
The only thing left to note is that $u(\sigma_i^{-1}(k),j) = c(i,j,k,\sigma)$.
\end{proof}

As we build towards the proof of \autoref{thm:symmetrization formula}, the strategy is to average each term on the right hand side of \autoref{prop:splitting formula for rog} over $S$. We start with the summand of the first term.

\begin{lemma}
\label{lem:average radius of gyration of Xi}
For $n > 1$ and for each $i \in \braces{1, \dotsc, \edgesE'}$, we have
\[
	\frac{1}{\# S} \sum_{\sigma \in S} \Rog(\positiongroup^{\sigma}_i)
	=
	\frac{n \pars{n+1} \pars{n-2}}{12 \pars{n-1}^2} \Rog(\displacementgroup_i) + \frac{n-2}{12 \, n} \norm{w'_i}^2.
\]
\end{lemma}

\begin{proof}
Using~\eqref{eq:DefRog1} from \autoref{defn:radius of gyration} we may write
\begin{equation}
	\frac{1}{\# S} \sum_{\sigma \in S} \Rog(\positiongroup^{\sigma}_i)
	=
	\frac{1}{\# S} \sum_{\sigma \in S} \frac{1}{2 \pars{n-1}^2} \sum_{1 \leq k,j < n} \norm{x^{\sigma}_{i,j} - x^{\sigma}_{i,k}}^2 .
\label{eq:getting started}
\end{equation}
Using~\autoref{lem:indicator variables}, we observe that
\[
	x^{\sigma}_{i,j} - x^{\sigma}_{i,k}
	=
	\sum_{\ell=1}^n \pars[\big]{ c(i,j,\ell,\sigma) - c(i,k,\ell,\sigma) } \, w_{i,\ell}.
\]
Further, the difference of the $c(i,-,-,-)$ terms can be rewritten in terms of a new indicator:
\[
u(i,j,k,\ell,\sigma) = \begin{cases}
+1, & \text{if $j > k$ and $k < \sigma^{-1}(\ell) \leq j$}, \\
-1, & \text{if $k > j$ and $j < \sigma^{-1}(\ell) \leq k$}, \\
0, & \text{otherwise}
\end{cases}
\]
which means that
\[
\begin{aligned}
\sum_{1 \leq k,j < n} \norm{x^{\sigma}_{i,j} - x^{\sigma}_{i,k}}^2 &=
\sum_{1 \leq k,j < n} \sum_{1 \leq \ell,m \leq n} u(i,j,k,\ell,\sigma) \, u(i,j,k,m,\sigma) \, \dot{w_{i,\ell}}{w_{i,m}},
\end{aligned}
\]
and hence that we can rewrite the right hand side of~\eqref{eq:getting started} as
\begin{equation}
	\sum_{1 \leq \ell, m \leq n}
	\dot{w_{i,\ell}}{w_{i,m}}  \frac{1}{2 \pars{n-1}^2} \sum_{1 \leq k,j < n}
	\left( \frac{1}{\# S} \sum_{\sigma \in S} u(i,j,k,\ell,\sigma) \, u(i,j,k,m,\sigma)  \right). \label{eq:alternate form}
\end{equation}
We note that the product of the $u(i,j,k,-,\sigma)$ is $+1$ if $\sigma^{-1}(\ell)$ and $\sigma^{-1}(m)$ are both between $k$ and $j$ and $0$ otherwise.
That is, the term in parentheses is the probability that $\sigma^{-1}(\ell)$ and $\sigma^{-1}(m)$ are between $j$ and $k$ when $\sigma$ is randomly selected in $S$. If $\ell = m$, then $\sigma^{-1}_i(\ell) = \sigma^{-1}_i(m)$ is uniformly distributed in $\braces{1, \dotsc, n}$ and so this probability is $\frac{|j-k|}{n}$. It can easily be checked that
\begin{equation}
	\sum_{1 \leq k,j < n} \frac{|j-k|}{n} = \frac{(n - 1)(n-2)}{3}. \label{eq:first sum}
\end{equation}
Otherwise, if $\ell \neq m$, then $\sigma^{-1}_i(\ell)$ and $\sigma^{-1}_i(m)$ are uniformly distributed among the $\binom{n}{2}$ pairs of numbers in $\braces{1, \dotsc, n}$, of which $\binom{|j-k|}{2}$ are between $j$ and $k$. Again, it is easy to check that
\begin{equation}
	\sum_{1 \leq k,j < n}
	\frac{\binom{|j-k|}{2}}{\binom{n}{2}} = \sum_{1 \leq k,j < n} \frac{(|j-k| -1) |j-k| }{\pars{n-1} \, n} = \frac{\pars{n-2}\pars{n-3}}{6}. \label{eq:second sum}
\end{equation}
Combining~\eqref{eq:getting started} with~\eqref{eq:alternate form},~\eqref{eq:first sum}, and~\eqref{eq:second sum} yields
\begin{equation}
\begin{aligned}
\frac{1}{\# S} \sum_{\sigma \in S} \Rog(\positiongroup^{\sigma}_i) &= \frac{n-2}{6 \pars{n-1}} \sum_{\ell=1}^{n} \dot{w_{i,\ell}}{w_{i,\ell}} + \frac{\pars{n-2} \pars{n-3}}{12 \pars{n-1}^2} \sum_{1 \leq \ell \neq m \leq n} \dot{w_{i,\ell}}{w_{i,m}} \\
&= \frac{\pars{n+1} \pars{n-2}}{12 \pars{n-1}^2} \sum_{\ell=1}^n \dot{w_{i,\ell}}{w_{i,\ell}} +
\frac{\pars{n-2} \pars{n-3}}{12 \pars{n-1}^2} \sum_{\ell, m = 1}^n \dot{w_{i,\ell}}{w_{i,m}}.
\end{aligned}
\label{eq:inner product formula}
\end{equation}
As noted above, if we think of $\edgegroup_i$ as the index set, $\displacementgroup_i$ is a weighted point cloud, with
\begin{equation}
\begin{aligned}
	\Rog(\displacementgroup_i)
	&=
	\frac{1}{2\,n^2} \sum_{\ell,m = 1}^n \norm{w_{i,\ell} - w_{i,m}}^2
	=
	\frac{1}{2\,n^2} \sum_{\ell,m = 1}^n
	\pars[\Big]{
		\norm{w_{i,\ell}}^2 + \norm{w_{i,m}}^2 - 2 \,\dot{w_{i,\ell}}{w_{i,m}}
	}
	\\
	&=
	\frac{1}{n} \sum_{\ell=1}^n \dot{w_{i,\ell}}{w_{i,\ell}}
	-
	\frac{1}{n^2} \sum_{\ell,m = 1}^n \dot{w_{i,\ell}}{w_{i,m}}.
\end{aligned} \label{eq:rog of wi}
\end{equation}
Solving this equation, we get $\sum_{\ell=1}^n \dot{w_{i,\ell}}{w_{i,\ell}} = n \Rog(\displacementgroup_i) + \frac{1}{n} \sum_{\ell = 1}^n \sum_{m = 1}^n \dot{w_{i,\ell}}{w_{i,m}}$.
Substituting into~\eqref{eq:inner product formula} and simplifying, we get
\[
\begin{aligned}
	\frac{1}{\# S} \sum_{\sigma \in S} \Rog(\positiongroup^{\sigma}_i)
	&= \frac{n \pars{n+1} \pars{n-2}}{12 \pars{n-1}^2} \Rog(\displacementgroup_i)
	+
	\frac{n-2}{12 \, n} \sum_{\ell = 1}^n \sum_{m = 1}^n \dot{w_{i,\ell}}{w_{i,m}}
	\\
	&=
	\frac{n \pars{n+1} \pars{n-2}}{12 \pars{n-1}^2} \Rog(\displacementgroup_i)
	+
	\frac{n-2}{12 \, n} \norm[\Bigg]{\sum_{\ell=1}^n w_{i,\ell}}^2.
\end{aligned}
\]
Since $\sum_{\ell=1}^n w_{i,\ell} = w'_i$, this completes the proof.
\end{proof}
In preparation for averaging the second and fourth terms in \autoref{prop:splitting formula for rog} over $S$, we define two new point clouds (with unit weights):
\begin{definition}
\label{def:center of mass cloud}
For each $i \in 0, \dotsc, \edgesE'$, we define the~$i$-th \emph{center of mass cloud} $M_i \co S \rightarrow \R^d$ by $\comcloud_i(\sigma) := \mu(\positiongroup^\sigma_i)$, where $\mu(\positiongroup^\sigma_i)$ is the center of mass of the point cloud $\positiongroup^\sigma_i$. We define the~$i$-th~\emph{parent cloud} to be the point cloud with  position function $P_i \co \{1, \dots, n-1\} \times S \to \R^d$ given by $\parentcloud_i(\sigma,j) = x^{\sigma}_{i,j}$. These point clouds have all weights equal to 1. 
\end{definition}
It is also convenient to define
\begin{definition}
\label{def:edge midpoints}
If $X'$ is an embedding of $\graphG'$, for each edge $\edge'_i$ of $\graphG'$, we define the midpoint $m'_i$ by
\[
	m'_i = \frac{1}{2} \pars[\big]{ x'_{\head(i)} + x'_{\tail(i)} }
\]
\end{definition}
The point clouds $\comcloud_i$ and $\parentcloud_i$ have a nice relationship with the midpoint $m'_i$:

\begin{figure}
	\includegraphics[
		trim = 0 0 0 0,
		clip = true,
		angle = 0,
		width=0.6\textwidth
	]{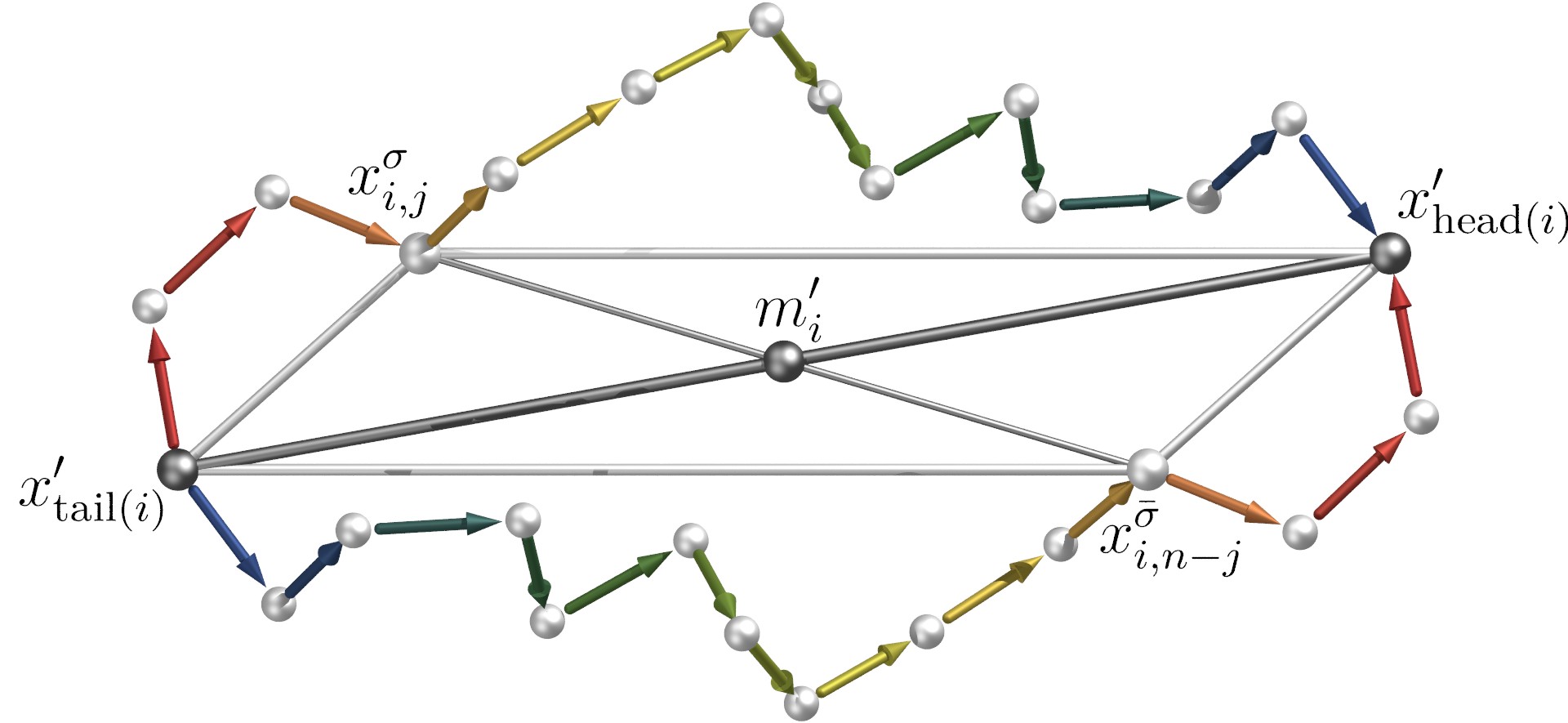}
	\caption{Here we see two paths from $x'_{\tail(i)}$ to $x'_{\head(i)}$ given by adding up the same set of edge displacements $w^\sigma_{i,j}$: once in the order given by the permutation $\sigma$ (top) and one by the reverse ordering $\bar\sigma$ (bottom). Because the vectors $x^{\sigma}_{i,j}-x'_{\tail(i)}$ and $x'_{\head(i)} - x^{\bar\sigma}_{i,n-j}$ are the sum of the same $j$ displacement vectors (colored orange and red in the picture), they are equal. This means that the four points $x'_{\tail(i)}, x^\sigma_{i,j}, x'_{\head(i)}, x^{\bar\sigma}_{i,n-j}$ form a parallelogram, shown in gray. The diagonals of the parallelogram meet at $m'_i$, which is therefore also the center of mass of the pair of points $x^\sigma_{i,j}$ and $x^{\bar\sigma}_{i,n-j}$.}
	\label{fig:midpoint-lemma}
\end{figure}

\begin{lemma}
\label{lem:center of mass of center of mass cloud}
For $i \in \braces{1, \dotsc, \edgesE'}$, we have $\mu(\parentcloud_i) = \mu(\comcloud_i) = m'_i$.
\end{lemma}

\begin{proof}
On the one hand, we have
\begin{equation}
	\mu(P_i)
	=
	\frac{1}{\# S} \frac{1}{n-1} \sum_{\sigma \in S} \sum_{j=1}^{n-1} x^\sigma_{i,j}
	=
	\frac{1}{\# S} \sum_{\sigma \in S}
	\underset{
		\comcloud_i = \mu(\positiongroup^\sigma_i)
	}{
		\underbrace{\pars[\bigg]{
			\frac{1}{n-1} \sum_{j=1}^{n-1} x^\sigma_{i,j}
		}}
	}
	=
	\mu(\comcloud_i)
	,
	\label{eq:muM}
\end{equation}
which is a special case of the law of iterated expectations.

On the other hand, for each $\sigma \in S$, there is a unique $\bar{\sigma} \in S$ so that $\sigma_k = \bar{\sigma}_k$ for $k \neq i$ and $\sigma_i$ is the reverse permutation from $\bar{\sigma}_i$: that is, that $\sigma_i(j) = \bar{\sigma}_i \pars{n+1-j}$. Because the map $S \to S$ defined by $\sigma \mapsto \bar{\sigma}$ is a bijection, we have
\begin{equation}
	\mu(P_i)
	=
	\frac{1}{\# S} \frac{1}{n-1} \sum_{\sigma \in S} \sum_{j=1}^{n-1}
	x^{\bar{\sigma}}_{i,n-j}
	=
	\frac{1}{\# S} \frac{1}{n-1} \sum_{\sigma \in S} \sum_{j=1}^{n-1}
	\frac{1}{2} \pars[\big]{ x^{\sigma}_{i,j} + x^{\bar{\sigma}}_{i,n-j} }
	.
	\label{eq:muP}
\end{equation}%
Using~\eqref{eq:definition of xsigmaij} and the fact that $x^{\sigma}_{0,j} = x'_j$ for all $j$, we compute
\[
\begin{aligned}
x^{\sigma}_{i,j} - x'_{\tail(i)} &= w_{\sigma(i,1)}+\cdots+w_{\sigma(i,j)} =  w_{\bar{\sigma}(i,n-j+1)} + \cdots + w_{\bar{\sigma}(i,n)} = x'_{\head(i)} - x^{\bar{\sigma}}_{i,n-j}\\
x'_{\head(i)} - x^{\sigma}_{i,j} &= w_{\sigma(i,j+1)} + \cdots + w_{\sigma(i,n)} = w_{\bar{\sigma}(i,1)} + \cdots + w_{\bar{\sigma}(i,n-j)} = x^{\bar{\sigma}}_{i,n-j} - x'_{\tail(i)}.
\end{aligned}
\]
so $x'_{\tail(i)}$, $x^\sigma_{i,j}$, $x'_{\head(i)}$, and $x^{\bar{\sigma}}_{i,n-j}$ are the vertices of a parallelogram, as shown in~\autoref{fig:midpoint-lemma}. Since the diagonals of a parallelogram intersect at their midpoints, we have
\[
	\frac{1}{2} \pars[\big]{x^{\sigma}_{i,j} + x^{\bar{\sigma}}_{i,n-j}}
	=
	\frac{1}{2} \pars[\big]{x'_{\tail(i)} + x'_{\head(i)}}
	=
	m_i'
	,
\]
regardless of the choice of $\sigma$ and~$j$. Using \eqref{eq:muM} and \eqref{eq:muP}, we now have $\mu(\comcloud_i) = \mu(\parentcloud_i) = m'_i$.
\end{proof}
We note that we have actually proved something much stronger than the statement of~\autoref{lem:center of mass of center of mass cloud}. For instance, this proof shows that the average center of mass of all~\emph{self-avoiding} walks from $p$ to $q$ is at the midpoint of $pq$, since the reverse of a self-avoiding walk is also self-avoiding. We are now ready to prove the main result of this section.

\begin{proposition}
\label{prop:average squared distance between centers of mass of edge groups}
Suppose $i \neq j$ and $i, j \in \braces{1, \dotsc, \edgesE'}$.
The average (over $\sigma \in S$) of the squared distance between $\mu(\positiongroup^\sigma_i)$ and $\mu(\positiongroup^\sigma_j)$ is given by
\[
\begin{aligned}
	\frac{1}{\# S} &\sum_{\sigma \in S} \norm[\big]{ \mu(\positiongroup^\sigma_i) - \mu(\positiongroup^\sigma_j) }^2
	=
 	\norm[\big]{m'_i - m'_j}^2 + \frac{n^2 \pars{n+1} }{12 \pars{n-1}^2} \pars[\big]{ \Rog(\displacementgroup_i) + \Rog(\displacementgroup_j) }
\end{aligned}
\]
\end{proposition}

\begin{proof}
The identity \eqref{eq:DefRog3} in~\autoref{defn:radius of gyration} implies
\[
	\frac{1}{\#S} \sum_{\sigma \in S} \norm[\big]{\mu(\positiongroup^\sigma_i) - \mu(\positiongroup^{\sigma}_j)}^2
	=
	\norm[\big]{ \mu(\comcloud_i) - \mu(\comcloud_j) }^2
	+
	\Rog\pars[\Big]{ \set[\big]{\mu(\positiongroup^\sigma_i) - \mu(\positiongroup^\sigma_j) }{ \sigma \in S } }
	.
\]
By \autoref{lem:center of mass of center of mass cloud}, we have
$\mu(\comcloud_i) = m'_i$ and $\mu(\comcloud_j) = m'_j$.

We now consider $\Rog \pars[\Big]{\set[\big]{\mu(\positiongroup^\sigma_i) - \mu(\positiongroup^\sigma_j)}{\sigma \in S }}$. First, think of the centers of mass $\mu(\positiongroup^\sigma_i)$ and $\mu(\positiongroup^\sigma_j)$ as random variables on the probability space $\sigma \in S$ (where each permutation is equally probable). By~\autoref{defn:big group of permutations}, these are independent random variables, because $S$ is a product of permutation groups. Therefore
\[
\Var(\mu(\positiongroup^\sigma_i) - \mu(\positiongroup^{\sigma}_j)) = \Var(\mu(\positiongroup^\sigma_i)) + \Var(-\mu(\positiongroup^{\sigma}_j)) = \Var(\mu(\positiongroup^\sigma_i)) + \Var(\mu(\positiongroup^\sigma_j)).
\]
or equivalently,
\[
	\Rog \pars[\Big]{\set[\big]{\mu(\positiongroup^\sigma_i) - \mu(\positiongroup^\sigma_j)}{\sigma \in S }} = \Rog(\comcloud_i) + \Rog(\comcloud_j)
\]
We now claim that for each $i \in \braces{1, \dotsc, \edgesE'}$, we have $\Rog(\comcloud_i) = \frac{n^2 \pars{n+1}}{12 \pars{n-1}^2} \Rog(\displacementgroup_i)$, which will complete the proof. We start by expressing~$\Rog(P_i)$ in terms of the properties of smaller point clouds using~\autoref{lem:splitting formula for Rog}, as we did in the proof of~\autoref{prop:splitting formula for rog}. 

Suppose that $\{ \varOmega_\sigma \mid \sigma \in S \}$ are weight functions on $\{1, \dots, n-1\} \times S$ indexed by permutations $\sigma \in S$, where $\varOmega_{\sigma}(x^{\sigma'}_{i,j}) = \KronDelta_{\sigma,\sigma'}$. We note $\varOmega := \sum_{\sigma \in S} \varOmega_{\sigma}$ has $\varOmega(x^{\sigma'}_{i,j}) = \sum_{\sigma \in S} \varOmega_{\sigma} (x^{\sigma'}_{i,j}) = 1$ for all $x^{\sigma'}_{i,j} \in P_i$. We compute $\abs{\varOmega_{\sigma}} = n-1$ and $\abs{\varOmega} = (n-1) \times \# S$. We also have $\Rog(P_i,\varOmega_{\sigma}) = \Rog(\positiongroup_i^{\sigma})$ and $\mu(P_i,\varOmega_{\sigma}) = \mu(X_i^{\sigma})$. 
Using~\autoref{lem:splitting formula for Rog}, we conclude
\[
\begin{aligned}
\Rog(\parentcloud_i) &= \frac{1}{\# S} \sum_{\sigma \in S} \Rog(\positiongroup_i^\sigma) + 
\frac{1}{2} \sum_{\sigma \in S} \sum_{\tau \in S} \frac{ \abs{\varOmega_{\sigma}} \abs{\varOmega_{\tau}}}{ \abs{\varOmega}^2 } \norm{\mu(X^\sigma_i) - \mu(X^\tau_i)}^2 \\
&= \frac{1}{\# S} \sum_{\sigma \in S} \Rog(\positiongroup_i^\sigma) + \frac{1}{2 (\# S)^2} \sum_{\sigma \in S} \sum_{\tau \in S} \norm{\mu(X^\sigma_i) - \mu(X^\tau_i)}^2 \\
&= \frac{1}{\# S} \sum_{\sigma \in S} \Rog(\positiongroup_i^\sigma) + \Rog(M_i)
\end{aligned}
\]
Since we've already computed $\Rog(M_i)$ in~\autoref{lem:average radius of gyration of Xi}, it remains to compute $\Rog(\parentcloud_i)$.

We claim that $\Rog(\parentcloud_i) = \frac{n \pars{n+1}}{6 \pars{n-1}} \Rog(\displacementgroup_i) + \frac{n-2}{12\,n} \norm{w'_i}^2$. We start with the statement that
\[
	\Rog(\parentcloud_i)
	=
	\frac{1}{\pars{n-1} \times (\#S)}
	\sum_{\sigma \in S} \sum_{j=1}^{n-1} \,
	\norm{x^{\sigma}_{i,j} - \mu(\parentcloud_i)}^2.
\]
Now using $\sum_{\ell=1}^n w_{i,\ell} = x'_{\head(i)} - x'_{\tail(i)}$ and~\autoref{lem:center of mass of center of mass cloud}, we have
\[
\begin{aligned}
	x^{\sigma}_{i,j} - \mu(\parentcloud_i)
	&=
	x'_{\tail(i)}
	+
	\sum_{\ell=1}^n c(i,j,k,\sigma) \, w_{i,\ell} - \frac{1}{2} \pars[\big]{ x'_{\head(i)} + x'_{\tail(i)} }
	\\
	&=
	\sum_{\ell=1}^n \pars[\Big]{ c(i,j,\ell,\sigma) - \frac{1}{2} } \, w_{i,\ell},
\end{aligned}
\]
so we may expand
\begin{equation}
	\norm{x^{\sigma}_{i,j} - \mu(\parentcloud_i)}^2
	=
	\sum_{\ell=1}^n \sum_{m=1}^n
	\pars[\Big]{c(i,j,\ell,\sigma) - \frac{1}{2} } \pars[\Big]{ c(i,j,m,\sigma) - \frac{1}{2} }
	\dot{w_{i,\ell}}{w_{i,m}}
	.
\end{equation}
As in the proof of~\autoref{lem:average radius of gyration of Xi}, we switch the order of summation, writing
\begin{equation}
	\sum_{\sigma \in S} \sum_{j=1}^{n-1}
	\norm{x^{\sigma}_{i,j} - \mu(\parentcloud_i)}^2
	=
	\sum_{\ell=1}^{n} \sum_{m=1}^{n}
	\dot{w_{i,\ell}}{w_{i,m}}
	\sum_{\sigma \in S} \sum_{j=1}^{n-1}
	\pars[\Big]{ c(i,j,\ell,\sigma) - \frac{1}{2} }
	\pars[\Big]{ c(i,j,m,\sigma) - \frac{1}{2} }
	.
\end{equation}
Now $c(i,j,k,\sigma) - \frac{1}{2}$ is equal to $\frac{1}{2}$ if $\sigma^{-1}_i(k) \leq j$ and $-\frac{1}{2}$ otherwise. Therefore, if
\[
u(i,j,\ell,m,\sigma) =
\begin{cases}
1, & \text{if $\sigma^{-1}_i(\ell), \sigma^{-1}_i(m)$ are both $\leq j$ or both $> j$}, \\
0, & \text{otherwise}
\end{cases}
\]
then
\[
	\pars[\Big]{ c(i,j,\ell,\sigma) - \frac{1}{2} }
	\pars[\Big]{ c(i,j,m,\sigma)    - \frac{1}{2} }
	=
	\frac{1}{2} u(i,j,\ell,m,\sigma) - \frac{1}{4}.
\]
If we think of $j$ and $\sigma$ as random variables uniformly distributed on $\braces{1, \dotsc, n-1}$ and $S$, then $\frac{1}{\pars{n-1} \times (\#S)} \sum_{\sigma} \sum_{j} u(i,j,\ell,m,\sigma)$ is the probability that $\sigma^{-1}_i(\ell)$ and $\sigma^{-1}_i(m)$ are either both~$\leq j$ or~both~$> j$.
If $\ell = m$, this probability is 1, regardless of the value of $j$.

If $\ell \neq m$, for any fixed $j$ there are $\binom{j}{2}$ unordered pairs $\ell,m \leq j$ and $\binom{n-j}{2}$ pairs $\ell,m > j$ among the $\binom{n}{2}$ pairs of numbers between $1$ and $n$. Thus the probability of this event when $j$ is randomly selected from $1, \dots, n-1$ is given by
\[
\frac{1}{n-1} \sum_{j=1}^{n-1} \frac{\binom{j-1}{2} + \binom{n-j}{2}}{\binom{n}{2}} =
\frac{2}{n-1} \frac{\binom{n}{3}}{\binom{n}{2}} = \frac{2 \pars{n-2}}{3 \pars{n-1}}
\]
where the middle step uses the ``hockey-stick'' identity for both sums of binomial coefficients (which are equal to each other).
We now know that
\[
\begin{aligned}
	\Rog(\parentcloud_i)
	&=
	\frac{1}{2} \sum_{\ell = 1}^n \dot{w_{i,\ell}}{w_{i,\ell}}
	+
	\frac{\pars{n-2}}{3 \pars{n-1}} \mathop{\sum\sum}\limits_{1 \leq \ell \neq m \leq n}
	\dot{w_{i,\ell}}{w_{i,m}}
	-
	\frac{1}{4}  \sum_{\ell=1}^{n} \sum_{m=1}^{n}  \dot{w_{i,\ell}}{w_{i,m}}
	\\
	&=
	\frac{n+1}{6 \pars{n-1}} \sum_{\ell = 1}^n \dot{w_{i,\ell}}{w_{i,\ell}}
	+
	\frac{n-5}{12 \pars{n-1}} \sum_{\ell=1}^{n} \sum_{m=1}^{n}  \dot{w_{i,\ell}}{w_{i,m}}
	\\
	&=
	\frac{n \pars{n+1}}{6 \pars{n-1}} \Rog(\displacementgroup_i)
	+
	\frac{n-2}{12\,n} \sum_{\ell=1}^{n} \sum_{m=1}^{n} \dot{w_{i,\ell}}{w_{i,m}}
	,
\end{aligned}
\]
where we used~\eqref{eq:rog of wi} in the last line, and note that this proves our claim. Now subtracting the result of~\autoref{lem:average radius of gyration of Xi} and simplifying completes the proof.
\end{proof}

\section{The degree-radius of gyration}

Chen and Zhang introduced the ``degree-Kirchhoff index''~\cite{Chen:2007dv} in the context of polymer science. This can be expressed in terms of a graph Laplacian of $\graphG$ with weights corresponding to those usually used in Riemannian geometry and Riemannian spectral graph theory~\cite{Chung1997}. We now consider a closely related quantity:

\begin{definition}
Given a graph $\graphG'$ with vertex set $\vertexgroup'$, the~\emph{degree weighted radius of gyration} of an embedding $X' \co \vertexgroup' \to \R^d$ is the radius of gyration of the weighted point cloud $(X',\deg)$ where $\deg(v'_i)$ is the degree of the corresponding vertex $\vertex'_i$. Since $\abs{\deg} = \sum_{i=1}^{\verticesV'} \deg(v'_i) = 2 \edgesE'$, this is 
\[
	\Rog(X',\deg) = \frac{1}{2 \edgesE'} \sum_{i = 1}^{\verticesV'}
	\deg(v'_i) \, \norm{ x'_i - \mu(X',\deg)}^2
	 \quad
	 \text{where}
	 \quad
	 \mu(X',\deg)
	 = \frac{1}{2 \edgesE'}\sum_{i = 1}^{\verticesV'}
	 \deg(v'_i) \, x'_i
	.
\]
\end{definition}
This weighted radius of gyration has a surprising connection with the ordinary radius of gyration:

\begin{proposition}
\label{prop:degree rog and rog}
Suppose that $X'$ is an embedding of the graph $\graphG'$ with displacements $W'$. Now let $M' = (m'_1, \dotsc, m'_{\edgesE'})$ be the point cloud consisting of the midpoints of the edges of $\graphG'$, weighted equally. Then
\[
	\mu(M')
	=
	\mu(X',\deg) \quad\text{and}\quad \Rog(M')
	=
	\Rog(X',\deg) - \frac{1}{4 \, \edgesE'} \sum_{i=1}^{\edgesE'\!\!} \norm{w'_i}^2.
\]
\end{proposition}

\begin{proof}
We start by observing that
\[
	\mu(M')
	=
	\frac{1}{\edgesE'} \sum_{i=1}^{\edgesE'\!\!} m'_i
	=
	\frac{1}{2 \, \edgesE'} \sum_{i=1}^{\edgesE'\!\!} \pars[\big]{ x'_{\head(i)} + x'_{\tail(i)} }
\]
Now each $x'_i$ appears a total of $\deg(v_i)$
times in this sum, either as the head or the tail of an edge. Further, the sum of the vertex degrees of a graph is well-known to be twice the total number of edges (even with loop or multiple edges, each edge contributes $+2$ to the sum of vertex degrees). Therefore,
\[
	\frac{1}{2 \, \edgesE'}
	\sum_{i=1}^{\edgesE'\!\!} \pars[\big]{ x'_{\head(\edge'_i)} + x'_{\tail(\edge')_i} }
	=
	\frac{1}{|\deg|} \sum_{i=1}^{\verticesV'} \deg(v'_i) \, x'_i
	=
	\mu(X',\deg)
	,
\]
proving the first part of the claim.

Now let's consider the second part. To save space, we let $h_i \ceq x'_{\head(i)}$ and $t_i \ceq x'_{\tail(i)}$ for the rest of this proof; note that $m'_i = \frac{1}{2}(h_i + t_i)$. Now $\Rog(M') = \frac{1}{2\,(\edgesE')^2} \sum_{i,j} \norm{m_i' - m_j'}^2$. Further, applying Euler's quadrilateral law~\cite{Kandall2002}, we have for each $i, j$:
\begin{equation}
\begin{aligned}
	\norm{m_i' - m_j'}^2 &=
	\frac{1}{4}
	\pars[\Big]{
		\norm{h_i - h_j}^2 + \norm{t_i - t_j}^2 + \norm{h_i - t_j}^2
		+
		\norm{h_j - t_i}^2} \\
	& \quad -\frac{1}{4} \pars[\Big]{\norm{h_i - t_i}^2 + \norm{h_j - t_j}^2
	}.
\end{aligned}
\label{eq:bimedian formula applied}
\end{equation}
The four positive terms on the right-hand side of~\eqref{eq:bimedian formula applied} are squared distances between vertices of~$X'$. Informally, each squared distance $\norm{x'_k - x'_\ell}^2$ occurs once for each pair of edges $e_i'$ where $\vertex_k'$ is incident to $\edge_i'$ and $\vertex_\ell'$ is incident to $\edge_j'$, and there are $\deg(v'_k) \deg(v'_\ell)$ such pairs. It is easy to check algebraically that this counts multi-edges and loop edges correctly. Using the definitions of $h_i$, $t_i$, $h_j$, and $t_j$, we get
\begin{align*}
	\MoveEqLeft
	\sum_{i=1}^{\edgesE'\!\!}\sum_{j=1}^{\edgesE'\!\!}
	\pars[\Big]{
		\norm{h_i - h_j}^2 + \norm{t_i - t_j}^2 + \norm{h_i - t_j}^2 + \norm{h_j - t_i}^2
	}
	=\\
	&=
	\sum_{i=1}^{\edgesE'\!\!} \sum_{j=1}^{\edgesE'\!\!} \sum_{k=1}^{\edgesE'\!\!} \sum_{\ell=1}^{\edgesE'\!\!}
	\pars[\big]{
		\KronDelta_{k,\head(i)} + \KronDelta_{k,\tail(i)}
	}
	\pars[\big]{
		\KronDelta_{\ell,\head(j)} + \KronDelta_{\ell,\tail(j)}
	}
	\norm{x'_k - x'_{\ell}}^2
	\\
	&=
	\sum_{k=1}^{\edgesE'\!\!} \sum_{\ell=1}^{\edgesE'\!\!}
	\pars[\bigg]{
		\sum_{i=1}^{\edgesE'\!\!} \pars[\big]{ \KronDelta_{k,\head(i)} + \KronDelta_{k,\tail(i)} }
	}
	\pars[\bigg]{
		\sum_{j=1}^{\edgesE'\!\!} \pars[\big]{  \KronDelta_{\ell,\head(j)} + \KronDelta_{\ell,\tail(j)} }
	}
	\norm{x'_k - x'_{\ell}}^2
	\\
	&=
	\sum_{k=1}^{\edgesE'\!\!} \sum_{\ell=1}^{\edgesE'\!\!} \deg(v'_k) \deg(v'_\ell) \, \norm{x'_k - x'_{\ell}}^2
	\\
	&= 2 \pars[\bigg]{\sum_{k=1}^{\edgesE'\!\!} \deg(v'_k)}^2 \Rog(X',\deg) = 8 \, (\edgesE')^2 \Rog(X',\deg)
	,
\end{align*}
where the last line follows from~\eqref{eq:DefRog1} in  \autoref{defn:radius of gyration} and the fact that $\sum_k \deg(v'_k) = 2 \, \edgesE'$. Hence, the four positive terms in~\eqref{eq:bimedian formula applied} contribute the term $\Rog(X',\deg)$ to $\Rog(M')$.

The second term in the expression for $\Rog(M')$ comes from the two negative terms on the right-hand side of~\eqref{eq:bimedian formula applied}, which are squared edgelengths.
Summing them yields 
\[
	-\sum_{i,j} \pars[\Big]{\norm{h_i - t_i}^2 + \norm{h_j - t_j}^2} = -2 \, \edgesE' \sum_{i} \norm{w'_i}^2,
\]
as desired.
\end{proof}

\section{The symmetrization formula}

We are now ready to average $\Rog(X^\sigma)$ over permutations $\sigma \in S$, thereby proving~\autoref{thm:symmetrization formula}. Surprisingly, the result only depends on the embedding $X'$, the number of subdivisions $n$, and the sum of the squares of the lengths of the edges in $X$, and has a relatively compact expression. We will be able to compute this~\emph{without knowing anything else about each $X^\sigma$}.

The claim in \autoref{thm:symmetrization formula} is that
\[
\begin{aligned}
	\frac{1}{\# S} \sum_{\sigma \in S} \Rog(X^\sigma)
	&=
	\Rog\pars[\Big]{X',\deg + \frac{2}{n-1}}
	+
	\frac{ \pars{n+1} \pars{2 \, \verticesV - n}}{12 \, \verticesV^2}
	\norm{W}^2
	-
	\frac{ \pars{n+1} \pars{2 \, \verticesV - 1}}{12 \, \verticesV^2} \norm{W'}^2 ,
\end{aligned}
\]
where $\norm{W}^2 = \sum_{i=1}^\edgesE \sum_{j=1}^{n} \norm{w_{i,j}}^2$ and $\norm{W'}^2 = \sum_{i=1}^{\edgesE'\!\!} \norm{w'_i}^2$.

\begin{proof}[Proof of \autoref{thm:symmetrization formula}]
\autoref{prop:splitting formula for rog} gives us a plan of attack. We have already symmetrized the first and third terms on the right hand side of \autoref{prop:splitting formula for rog}
in \autoref{lem:average radius of gyration of Xi} and \autoref{prop:average squared distance between centers of mass of edge groups}, respectively. The second term is invariant under the action of $S$.

So we now focus on the fourth term. Swapping the order of summation on the left hand side, we see that
\[
	\frac{1}{\# S} \sum_{\sigma \in S} \sum_{i=1}^{\edgesE'\!\!} \norm{\mu(\positiongroup_i^\sigma) - \mu(X')}^2 =
	\sum_{i=1}^{\edgesE'\!\!}
	\pars[\bigg]{
		\frac{1}{\# S} \sum_{\sigma \in S} \norm{
				\mu(\positiongroup_i^\sigma) -  \mu(X')
		}^2
	}
	.
\]
The inner sum is the average squared norm of points in $\{ \mu(\positiongroup_i^\sigma) - \mu(X') \mid \sigma \in S\}$, which is the translation of $\comcloud_i$ by $-\mu(X')$. Using~\eqref{eq:DefRog3} from \autoref{defn:radius of gyration}, we can write this sum in terms of $\Rog(M_i - \mu(X')) = \Rog(M_i)$:
\[
	\frac{1}{\# S} \sum_{\sigma \in S} \norm{\mu(\positiongroup_i^\sigma) - \mu(X')}^2
	=
	\Rog(\comcloud_i - \mu(X'))
	+
	\norm{ \mu(\comcloud_i) - \mu(X') }^2
	=
	\Rog(\comcloud_i)
	+
	\norm{ \mu(\comcloud_i) - \mu(X') }^2
	.
\]
In the proof of~\autoref{prop:average squared distance between centers of mass of edge groups}, we showed that $\Rog(\comcloud_i) = \frac{n^2 \pars{n+1}}{12 \pars{n-1}^2} \Rog(\displacementgroup_i)$.
Applying~\autoref{lem:center of mass of center of mass cloud} shows that $\sum_{i=1}^{\edgesE'\!\!} \norm{\mu(\comcloud_i) - \mu(X')}^2 = \sum_{i=1}^{\edgesE'\!\!} \norm{m'_i - \mu(X')}^2$. As above, we can write this in terms of the radius of gyration of the point cloud $\set{m'_i - \mu(X')}{i \in \braces{1,\dotsc,\edgesE'}}$, which is a translation of the edge midpoint cloud $M'$. Using~\autoref{prop:degree rog and rog}, we get
\begin{equation*}
	\frac{1}{\edgesE'} \sum_{i=1}^{\edgesE'\!\!} \norm{ m'_i - \mu(X') }^2
	=
	\Rog(M') + \norm{\mu(M') - \mu(X')}^2
	=
	\Rog(M') + \norm{\mu(X',\deg) - \mu(X')}^2
	.
\end{equation*}
We have now established that
\[
	\frac{1}{\# S} \sum_{\sigma \in S} \sum_{i=1}^{\edgesE'\!\!} \norm{\mu(\positiongroup_i^\sigma) - \mu(X')}^2
	=
	\edgesE' \pars[\Big]{ \Rog(M') + \norm{\mu(X',\deg) - \mu(X')}^2 }
	+
	\frac{n^2 \pars{n+1}}{12 \pars{n-1}^2} \sum_{i=1}^{\edgesE'\!\!} \Rog(\displacementgroup_i).
\]
We now return to the statement of~\autoref{prop:splitting formula for rog} and symmetrize:
\[
\begin{aligned}
	\frac{1}{\#S} \sum_{\sigma \in S} \Rog(X^\sigma)
	&=
	\frac{n-1}{\verticesV} \sum_{i=1}^{\edgesE'\!\!}  \frac{1}{\# S} \sum_{\sigma \in S} \Rog(\positiongroup^\sigma_i)
	+ \frac{\pars{n-1}^2}{2 \,\verticesV^2}
	\sum_{i = 1}^{\edgesE'\!\!} \sum_{j = 1}^{\edgesE'\!\!}
	\frac{1}{\# S}
	\sum_{\sigma \in S} \norm{\mu(\positiongroup^\sigma_i) - \mu(\positiongroup^\sigma_j)}^2
	+
	\\
	&\qquad
	+ \frac{\verticesV'}{\verticesV} \Rog(X')
	+ \frac{\pars{n-1} \, \verticesV'}{\verticesV^2} \frac{1}{\# S} \sum_{\sigma \in S} \sum_{i=1}^{\edgesE'\!\!} \norm{\mu(\positiongroup^\sigma_i) - \mu(X')}^2
.
\end{aligned}
\]
Using~\autoref{lem:average radius of gyration of Xi},~\autoref{prop:average squared distance between centers of mass of edge groups}, the fact that $\sum_{i \neq j} \norm{m'_i - m'_j}^2 = 2\,(\edgesE') \Rog(M')$, and~\autoref{prop:degree rog and rog}, we can now expand the above and collect terms, observing that $\sum_{i=1}^{\edgesE'\!\!} \Rog(\displacementgroup_i)$ and $\sum_{i=1}^{\edgesE'\!\!} \norm{w'_i}^2$ occur in several terms.
Simplifying very carefully and remembering that $\verticesV = \pars{n-1} \, \edgesE' + \verticesV'$ yields
\begin{equation}
\begin{aligned}
	\frac{1}{\#S} \sum_{\sigma \in S} \Rog(X^\sigma)
	&=
	\frac{n \pars{n+1} \pars{ 2 \, \verticesV-n } }{12 \,\verticesV^2}
	\sum_{i=1}^{\edgesE'\!\!} \Rog(\displacementgroup_i)
	-
	\frac{n^2-1}{6 \, n \, \verticesV} \sum_{i=1}^{\edgesE'\!\!} \norm{w'_i}^2
	\\
	&+
	\frac{\verticesV - \verticesV'}{\verticesV} \Rog(X',\deg) + \frac{\verticesV'}{\verticesV} \Rog(X')
	+
	\frac{\verticesV' \, (\verticesV - \verticesV')}{\verticesV^2} \norm{\mu(X',\deg) - \mu(X')}^2
	.
\end{aligned} \label{eq:intermediate form}
\end{equation}
We will now combine the last three terms on the right-hand side of~\eqref{eq:intermediate form} into a single weighted radius of gyration using~\autoref{lem:splitting formula for Rog}. Recall that multiplying the weight function in any radius of gyration or center of mass formula by a constant has no effect on the result. Therefore, if we define the weight function $\varOmega'_1 : \vertexgroup' \to \R^+$ by $\varOmega'_1(v'_k) = \frac{n-1}{2} \deg(v'_k)$, we have $\Rog(X',\deg) = \Rog(X',\varOmega'_1)$ and $\mu(X',\deg) = \mu(X',\varOmega'_1)$.
Further, 
\[
\abs{\varOmega'_1} = \sum_{k = 1}^{\verticesV'} \frac{\pars{n-1} \deg(v'_k)}{2} = (n-1) \edgesE' = \verticesV - \verticesV'.
\]
On the other hand, if we define $\varOmega'_2 : \vertexgroup' \to \R^+$ by $\varOmega'_2(v'_k) = 1$, we have $\Rog(X',\varOmega'_2) = \Rog(X')$. Further, $\abs{\varOmega'_2} = \verticesV'$, so if $\varOmega' = \varOmega'_1 + \varOmega'_2$, then $\abs{\varOmega'} = \verticesV$. Applying~\autoref{lem:splitting formula for Rog} we then have
\[
\begin{aligned}
&\left(\frac{\verticesV - \verticesV'}{\verticesV} \Rog(X',\deg) + \frac{\verticesV'}{\verticesV} \Rog(X') \right)
+
\frac{\verticesV' \, (\verticesV - \verticesV')}{\verticesV^2} \norm{\mu(X',\deg) - \mu(X')}^2 = \\
&\qquad\qquad= \sum_{i=1}^2 \frac{\abs{\varOmega'_i}}{\abs{\varOmega'}} \Rog(X',\varOmega'_i) + 
\frac{1}{2} \sum_{i=1}^{2} \sum_{j=1}^{2} \frac{\abs{\varOmega'_i} \abs{\varOmega'_j}}{\abs{\varOmega'}^2} \norm{\mu(X',\varOmega'_i) - \mu(X',\varOmega'_j)}^2 \\
&\qquad\qquad= \Rog(X',\varOmega') = \Rog(X',\frac{n-1}{2} \deg + 1) = \Rog(X',\deg + \frac{2}{n-1}).
\end{aligned}
\]
Further, since $w'_i = \sum_{j=1}^n w_{i,j}$, we can use
\[
	\sum_{i=1}^{\edgesE'\!\!} \Rog(\displacementgroup_i)
	=
	\sum_{i=1}^{\edgesE'\!\!}
	\pars[\Bigg]{
		\frac{1}{n} \sum_{j=1}^{n} \norm{ w_{i,j} }^2
		-
		\norm[\bigg]{ \frac{1}{n} \sum_{j=1}^n w_{i,j} }^2
	}
	=
	\frac{1}{n} \sum_{i=1}^{\edgesE'\!\!} \sum_{j=1}^{n} \norm{w_{i,j}}^2 - \frac{1}{n^2} \sum_{i=1}^{\edgesE'\!\!}\norm{ w'_i }^2
\]
to rewrite the first two terms on the right hand side of~\eqref{eq:intermediate form} in terms of $\norm{W}^2$ and $\norm{W'}^2$, producing the claimed formula for average radius of gyration and completing the proof of \autoref{thm:symmetrization formula}.
\end{proof}

\autoref{thm:symmetrization formula} is a generalization of~\cite[Proposition~6.5]{Cantarella2012d}, which covers the special case where $\graphG'$ has one edge joining two vertices.
Translated to the notation used here, it states that
\[
	\frac{1}{\#S} \sum_{\sigma \in s} \Rog( X )
	=
	\frac{n+2}{12 \pars{n+1}} \pars[\bigg]{ \sum_{j=1}^n \norm{w_{1,j}}^2 + \norm{w_1'}^2}
\]
We can reproduce this with our formula from \autoref{thm:symmetrization formula}. Note that
\[
	\Rog\pars[\Big]{ X',\deg + \frac{2}{n-1} }
	= \frac{1}{4} \norm{x_2' - x_1'}^2
	= \frac{1}{4} \norm{w_1'}^2
	.
\]%
Simplifying the statement of~\autoref{thm:symmetrization formula} by using $\verticesV = n +1$, we see that for any collection of edges, the expected radius of gyration is $\frac{n+2}{12 \pars{n+1}} (\norm{W}^2 + \norm{W'}^2)$, as expected.

Moreover, when $x_1' = x_2'$, this reproduces the expected radius of gyration and pairwise edge correlations already derived for various random polygon models (see~\cite[Proposition~7.2 and Corollary~7.3]{Cantarella2012d}, and compare to~\cite[Lemma~2 and Theorem~6]{Zirbel2012} and~\cite[equation (7)]{Grosberg2008}). We also verified the formula numerically for a variety of graphs, and encourage the reader to do the same.


\section*{Acknowledgments}
We are very grateful to many colleagues for helpful discussions, especially Sebastian Caillaut, Azim Sharapov, Tetsuo Deguchi, and Erica Uehara. We thank the National Science Foundation (DMS--2107700 to Shonkwiler) and the Simons Foundation (\#524120 to Cantarella, \#709150 to Shonkwiler) for their support.

\bibliography{mendeley-export,contractionfactors}

\appendix
\section{Proof of~\autoref{lem:splitting formula for Rog}}

\begin{proof}
We let $r_i^2 \ceq \Rog(X,\varOmega_i)$ and $\mu_i \ceq \mu(X,\varOmega_i)$, while $\varOmega_k = \varOmega(x_k)$ and $\varOmega_{i,k} = \varOmega_i(x_k)$. 
\begin{align*}
	\MoveEqLeft
	2 \abs{\varOmega}^2
	\Rog(X,\varOmega)
	=
	\sum_{j=1}^m \sum_{\ell=1}^{n}
	\norm{x_{k} - x_{\ell}}^2 \, \varOmega_{k} \, \varOmega_{\ell}
	=
	\sum_{i=1}^m \sum_{k=1}^{n}  \sum_{j=1}^m \sum_{\ell=1}^{n}
	\norm{x_{k} - x_{\ell}}^2 \, \varOmega_{i,k} \, \varOmega_{j,\ell}
	\\
	&=
	\sum_{i=1}^m \sum_{k=1}^{n}  \sum_{j=1}^m \sum_{\ell=1}^{n}
	\pars[\Big]{
		\norm{x_{k}}^2 - 2\inner{x_{k}, x_{\ell}} + \norm{x_{\ell}}^2
	}\, \varOmega_{i,k} \, \varOmega_{j,\ell}
	\\
	&=
	2 \sum_{j=1}^m \pars[\Bigg]{
		\underset{\abs{\varOmega_j}}{\underbrace{\sum_{\ell=1}^{n_j}  \varOmega_{j,\ell} }}
	}
	\pars[\Bigg]{
		\sum_{i=1}^m
		\underset{
			\abs{\varOmega_i} \, r_i^2
			+
			\abs{\varOmega_i} \norm{\mu_i}^2
		}{\underbrace{
		\sum_{k=1}^{n}
		\norm{x_{k}}^2  \, \varOmega_{i,k}
		}}
	}
	- 2 \sum_{i=1}^m \sum_{j=1}^m
	\inner[\Bigg]{
		\underset{\abs{\varOmega_i} \, \mu_i}{\underbrace{
			\sum_{k=1}^{n}
			x_{k} \, \varOmega_{i,k}
		}}
		,
		\underset{\abs{\varOmega_j} \, \mu_j}{\underbrace{
			\sum_{\ell=1}^{n}
			x_{\ell} \, \varOmega_{j,\ell}
		}}
	}
	\\
	&=
	2 \,  \abs{\varOmega} \sum_{i=1}^m  \abs{\varOmega_i} \, r_i^2
	+
	\sum_{i=1}^m \sum_{j=1}^m
	\abs{\varOmega_i} \abs{\varOmega_j} \, \pars[\big]{
		\norm{\mu_i}^2
		+
		\norm{\mu_j}^2
	}
	- 2 \sum_{i=1}^m \sum_{j=1}^m
	\abs{\varOmega_i} \abs{\varOmega_j } \inner{ x_{k} , x_{\ell} }
	\\
	&=
	2 \,  \abs{\varOmega} \sum_{i=1}^m  \abs{\varOmega_i} \, r_i^2
	+
	\sum_{i=1}^m \sum_{j=1}^m \abs{\varOmega_i} \abs{\varOmega_j} \norm{\mu_i - \mu_j}^2
\end{align*}
\end{proof}

\end{document}